\documentclass[12pt,a4paper]{article}
\usepackage{amsmath,amssymb,amsbsy,amscd,amsthm}

\theoremstyle{plain}
\newtheorem{thm}{Theorem}[section]
\newtheorem{lem}[thm]{Lemma}
\newtheorem{prop}[thm]{Proposition}
\newtheorem{cor}[thm]{Corollary}

\newcommand{\trd }{\mathop{\rm trans.deg}\nolimits}
\newcommand{\supp }{\mathop{\rm supp}\nolimits}
\newcommand{\New }{\mathop{\rm New}\nolimits}
\newcommand{\spec }{\mathop{\rm Spec}\nolimits}
\newcommand{\uu}{{\bf u}}
\newcommand{\vv}{{\bf v}}
\newcommand{\w}{{\bf w}}
\newcommand{\id	}{{\rm id}}
\newcommand{\degw}{\deg_{\w}}
\newcommand{\degv}{\deg_{\vv}}
\newcommand{\degu}{\deg_{\uu}}
\newcommand{\zs}{\{ 0\} }
\newcommand{\sm}{\setminus}
\newcommand{\Ga}{{\bf G}_{\rm a}}
\newcommand{\C}{{\bf C}}
\newcommand{\R}{{\bf R}}
\newcommand{\Q}{{\bf Q}}
\newcommand{\Z}{{\bf Z}}
\newcommand{\x}{{\bf x}}
\newcommand{\kx}{k[{\bf x}]}
\newcommand{\kxb}{k[\overline{\bf x}]}

\begin{document}

\title{Initial forms of stable invariants for additive group actions}

\author{Shigeru Kuroda\thanks{
Partly supported by the Grant-in-Aid for 
Young Scientists (B) 24740022, 
Japan Society for the Promotion of Science. }}


\date{}

\maketitle

\begin{abstract}
The Derksen--Hadas--Makar-Limanov theorem (2001) says that 
the invariants for nontrivial actions of the additive group 
on a polynomial ring have no intruder. 
In this paper, 
we generalize this theorem to the case of stable invariants. 
\end{abstract}

\section{Introduction}
\label{sect:intro}
\setcounter{equation}{0}

Throughout this paper, 
let $k$ be a domain unless otherwise stated, 
and $\kx =k[x_1,\ldots ,x_n]$ 
the polynomial ring in $n$ variables over $k$, 
where $n\geq 1$. 
For each 
\begin{equation}\label{eq:f}
f=\sum _{i_1,\ldots ,i_n}
u_{i_1,\ldots ,i_n}x_1^{i_1}\cdots x_n^{i_n}\in \kx 
\end{equation}
with $u_{i_1,\ldots ,i_n}\in k$, 
we define $\supp (f)=\{ (i_1,\ldots ,i_n)\mid 
u_{i_1,\ldots ,i_n}\neq 0\} $. 
We call the convex hull $\New (f)$ of 
$\supp (f)$ in $\R ^n$ the {\it Newton polytope} of $f$. 
A vertex $(i_1,\ldots ,i_n)$ of $\New (f)$ 
is called an {\it intruder} of $f$ 
if $i_l\neq 0$ for $l=1,\ldots ,n$.

Let $\Ga =\spec k[z]$ be the additive group, 
where $z$ is an indeterminate. 
A homomorphism $\sigma :\kx \to \kx [z]=\kx \otimes _kk[z]$ 
of $k$-algebras is called a $\Ga $-{\it action} on $\kx $ 
if $\pi \circ \sigma =\id _{\kx }$, 
and the diagram 
$$
\begin{CD}
\kx @>\sigma >>\kx \otimes _kk[z]\\
@V\sigma VV 
@VV\sigma \otimes \id _{k[z]}V \\
\kx \otimes _kk[z] @>>\id _{\kx }\otimes \mu > \kx \otimes _kk[z]\otimes _kk[z]
\end{CD}
$$
commutes. 
Here, 
$\pi :\kx [z]\to \kx $ 
and $\mu :k[z]\to k[z]\otimes _kk[z]$ 
are the homomorphisms of $\kx $-algebras 
and $k$-algebras defined by 
$\pi (z)=0$ and $\mu (z)=z\otimes 1+1\otimes z$, 
respectively. 
We note that $\sigma (f)=f$ if and only if $\sigma (f)$ 
belongs to $\kx $ for $f\in \kx $. 
When this is the case, 
we call $f$ an {\it invariant} for $\sigma $. 
The set $\kx ^{\Ga }:=\sigma ^{-1}(\kx )$ 
of invariants for $\sigma $ forms a $k$-subalgebra of $\kx $. 
The $\Ga $-action defined by the inclusion map 
$\kx \to \kx [z]$ is called a {\it trivial} $\Ga $-action. 
A $\Ga $-action is trivial if and only if $\kx ^{\Ga }=\kx $.

The Derksen--Hadas--Makar-Limanov 
theorem~\cite[Theorem 3.1]{DHM} 
says that the invariants 
for any nontrivial $\Ga$-action on $\kx $ have no intruder. 
This theorem implies that, 
if $f_1,\ldots ,f_n\in \kx $ 
satisfy $k[f_1,\ldots ,f_n]=\kx $, 
then no element of $k[f_2,\ldots ,f_n]$ 
has an intruder \cite[Corollary 3.2]{DHM}.

The purpose of this paper is to present 
``stable versions" of the results above. 
For $m\geq n$, 
let $\kxb =k[x_1,\ldots ,x_m]$ be the polynomial ring 
in $m$ variables over $k$. 
We call $f\in \kx $ 
a {\it stable $\Ga $-invariant} of $\kx $ 
if there exist $m\geq n$ and a $\Ga $-action on $\kxb$ 
for which $\kxb ^{\Ga }$ contains $f$, 
but does not contain $\kx $. 
If $f\in \kx $ is an invariant 
for some nontrivial $\Ga $-action on $\kx $, 
then $f$ is a stable $\Ga $-invariant by definition. 
However, 
it is not known whether the converse holds in general 
(cf.~Section~\ref{sect:rmk}).

We generalize the Derksen--Hadas--Makar-Limanov theorem 
as follows.

\begin{thm}\label{thm:Newton}
No stable $\Ga $-invariant of $\kx $ has an intruder. 
\end{thm}

This theorem is a consequence of 
a more general result as follows. 
Let $\Gamma $ be a 
{\it totally ordered additive group}, 
i.e., 
an additive group 
equipped with a total ordering such that 
$\alpha \leq \beta $ implies 
$\alpha +\gamma \leq \beta +\gamma $ 
for each $\alpha ,\beta ,\gamma \in \Gamma $. 
For example, 
$\R $ is a totally ordered additive group 
for the standard ordering. 
Take any 
$\w =(w_1,\ldots ,w_n)\in \Gamma ^n$. 
We denote $a\cdot \w =a_1w_1+\cdots +a_nw_n$ 
for $a=(a_1,\ldots ,a_n)\in \Z ^n$. 
For each $f\in \kx \sm \zs $, 
we define the $\w $-{\it degree} $\degw f$ 
and $\w $-{\it initial form} $f^{\w }$ 
by 
$$
\degw f=\max \{ a\cdot \w \mid 
a\in \supp (f)\}\text{ \ and \ }
f^{\w }=\sum _{i_1,\ldots ,i_n}u'_{i_1,\ldots ,i_n}
x_1^{i_1}\cdots x_n^{i_n}, 
$$
where $u'_{i_1,\ldots ,i_n}=u_{i_1,\ldots ,i_n}$ 
if $(i_1,\ldots ,i_n)\cdot \w =\degw f$, 
and $u'_{i_1,\ldots ,i_n}=0$ otherwise. 
When $f=0$, 
we define $\degw f=-\infty $ and $f^{\w }=0$. 
Then, 
it holds that 
\begin{equation}\label{eq:deg fg}
\degw fg=\degw f+\degw g\quad
\text{and}\quad (fg)^{\w }=f^{\w }g^{\w }
\end{equation}
for each $f,g\in \kx $. 
We remark that 
$f\in \kx \sm \zs $ has no intruder if and only if, 
for each $\w \in \R ^n$ with $f^{\w }$ a monomial, 
there exists $1\leq i\leq n$ 
such that $x_i$ does not divide $f^{\w }$.

Now, 
let $\phi :\kx \to \kxb [z]$ 
be a homomorphism of $k$-algebras 
such that $\phi (\kx )$ 
is not contained in $\kxb $, 
and $\w $ an element of $\Gamma ^n$ 
which satisfies the following condition, 
where $p_i(z):=\phi (x_i)\in \kxb [z]$ for each $i$:

\smallskip 

\noindent ($\ast $) There exists 
$\vv \in \Gamma ^m$ such that 
$p_1(0)^{\vv },\ldots ,p_n(0)^{\vv }$ 
are algebraically independent over $k$, 
and $\degv p_i(0)=w_i$ for $i=1,\ldots ,n$.

\smallskip

In this situation, 
$\kx ^{\phi }:=\phi ^{-1}(\kxb )$ 
is a proper $k$-subalgebra of $\kx $. 
The following theorem 
will be proved in the next section.

\begin{thm}\label{thm:main}
Let $\phi :\kx \to \kxb [z]$ 
and $\w\in \Gamma ^n$ be as above, 
and $S\subset \kx \sm \zs $ 
such that $\trd _kk[S]=n$. 
Then, 
there exists $g\in S$ such that 
$g$ does not divide $f^{\w }$ 
for any $f\in \kx ^{\phi }\sm \zs $. 
\end{thm}

If $f$ is a stable $\Ga $-invariant of $\kx $, 
then there exist $m\geq n$ 
and a $\Ga $-action $\sigma :\kxb \to \kxb [z]$ 
such that $\sigma ^{-1}(\kxb )$ contains $f$, 
but does not contain $\kx $. 
We define $\phi =\sigma |_{\kx }$. 
Then, 
$\phi (\kx )=\sigma (\kx )$ 
is not contained in $\kxb $. 
We claim that 
any $\w \in \Gamma ^n$ satisfies ($\ast $) 
for this $\phi $. 
In fact, 
since $\pi \circ \sigma =\id _{\kxb }$, 
we have 
$p_i(0)=\pi (\sigma (x_i))=x_i$ 
for each $i$. 
Hence, 
($\ast $) holds for 
$\vv =(\w ,0,\ldots ,0)\in \Gamma ^m$. 
Clearly, 
$f$ belongs to 
$\sigma ^{-1}(\kxb )\cap \kx =\phi ^{-1}(\kxb )=\kx ^{\phi }$. 
Therefore, 
we obtain the following theorem 
as a consequence of Theorem~\ref{thm:main}.

\begin{thm}\label{thm:stb inv}
Let $f$ be a nonzero stable $\Ga $-invariant 
of $\kx $, 
and $S\subset \kx \sm \zs $ 
such that $\trd _kk[S]=n$. 
Then, 
$f^{\w }$ is not divisible by an element of $S$ 
for each $\w \in \Gamma ^n$. 
\end{thm}

Since $S=\{ x_1,\ldots ,x_n\} $ satisfies $\trd _kk[S]=n$, 
Theorem~\ref{thm:Newton} 
follows from Theorem~\ref{thm:stb inv} 
by virtue of the above remark on intruders. 
As another application of Theorem~\ref{thm:stb inv}, 
we obtain the following theorem.

\begin{thm}\label{thm:coords}
Let $m\geq n$ and 
$f_1,\ldots ,f_m\in \kxb $ 
be such that $k[f_1,\ldots ,f_m]=\kxb $ 
and $\kx $ is not contained in $k[f_2,\ldots ,f_m]$, 
and let $S\subset \kx \sm \zs $ be 
such that $\trd _kk[S]=n$. 
Then, 
for each $\w \in \Gamma ^n$, 
there exists $g\in S$ 
such that $g$ does not divide $f^{\w }$ 
for any $f\in k[f_2,\ldots ,f_m]\cap \kx \sm \zs $. 
\end{thm}

In fact, 
we have $\kxb ^{\Ga }=k[f_2,\ldots ,f_m]$ 
for the $\Ga $-action on $\kxb $ defined by 
$f_1\mapsto f_1+z$ 
and $f_i\mapsto f_i$ for each $i\geq 2$. 
Hence, 
every element of $k[f_2,\ldots ,f_m]\cap \kx $ 
is a stable $\Ga $-invariant of $\kx $ 
unless $\kx $ is contained in $k[f_2,\ldots ,f_m]$.

We call $f\in \kx $ a {\it coordinate} of $\kx $ 
if there exist $f_2,\ldots ,f_n\in \kx $ 
such that $k[f,f_2,\ldots ,f_n]=\kx $, 
and a {\it stable coordinate} of $\kx $ 
if there exists $m\geq n$ such that 
$f$ is a coordinate of $\kxb $. 
By definition, 
every coordinate of $\kx $ 
is a stable coordinate of $\kx $. 
Since $k$ is a domain, 
the converse is clear if $n=1$. 
If $n=2$, however, 
not every stable coordinate of $\kx $ 
is a coordinate of $\kx $ 
(cf.~\cite{BD}; see also Section~\ref{sect:rmk}).

Assume that $n\geq 2$, 
and let $f_1$ be a stable coordinate of $\kx $. 
Then, 
there exist $m\geq n$ 
and $f_2,\ldots ,f_m\in \kxb $ 
such that $k[f_1,\ldots ,f_m]=\kxb $. 
Since $n\geq 2$, 
we see that 
$\kx $ is not contained in 
$k[f_1]=\bigcap _{i=2}^mk[\{ f_j\mid j\neq i\} ]$. 
Hence, 
there exists $2\leq i_0\leq m$ such that 
$\kx $ is not contained in $k[\{ f_j\mid j\neq i_0\} ]$. 
Since $f_1$ belongs to 
$k[\{ f_j\mid j\neq i_0\} ]\cap \kx \sm \zs $, 
we obtain the following corollary to 
Theorem~\ref{thm:coords}.

\begin{cor}\label{cor:stb coord}
Assume that $n\geq 2$. 
Let $f$ be a stable coordinate of $\kx $, 
and let $S\subset \kx \sm \zs $ 
be such that $\trd _kk[S]=n$. 
Then, 
$f^{\w }$ is not divisible by an element of $S$ 
for each $\w \in \Gamma ^n$. 
\end{cor}

It is possible that an element of $\kx $ 
which is not a coordinate of $\kx $ 
becomes a coordinate of $k_0[\x ]$, 
where $k_0$ is the field of fractions of $k$. 
Hence, 
an element of $\kx $ 
which is not a stable coordinate of $\kx $ 
can be a stable coordinate of $k_0[\x ]$. 
We note that the conclusion of 
Corollary~\ref{cor:stb coord} holds 
if only $f\in \kx $ 
is a stable coordinate of $k_0[\x ]$. 
Similarly, 
the conclusion of Theorem~\ref{thm:stb inv} 
holds if only $f\in \kx \sm \zs $ 
is a stable $\Ga $-invariant of $k_0[\x ]$.

The author would like to thank 
Prof.\ Hideo Kojima for pointing out 
Freudenburg's paper mentioned in 
Section~\ref{sect:rmk}, 
and Prof.\ Amartya K. Dutta and Prof.\ Neena Gupta 
for the invaluable remarks 
in Section~\ref{sect:rmk}.

\section{Proof of Theorem~\ref{thm:main}}
\setcounter{equation}{0}

For a domain $R$ and a subring $S$ of $R$, 
we say that $S$ is {\it factorially closed} in $R$ 
if $ab\in S$ implies $a,b\in S$ 
for each $a,b\in R\sm \zs $. 
We remark that $R$ 
is algebraically closed 
and factorially closed in the polynomial ring 
$R[x_1,\ldots ,x_n]$ if $R$ is a domain.

Let $R$ and $R'$ be domains, 
and $\psi :R\to R'[x_1,\ldots ,x_n]$ 
a homomorphism of rings. 
Then, 
the following lemma holds for $R^{\psi }:=\psi ^{-1}(R')$.

\begin{lem}\label{lem:closed}
If $\psi $ is injective, 
then $R^{\psi }$ is algebraically closed 
and factorially closed in $R$. 
\end{lem}
\begin{proof}
Since $R'$ is algebraically closed 
and factorially closed in $R'[x_1,\ldots ,x_n]$, 
we see that $\psi (R)\cap R'$ 
is algebraically closed and factorially closed
in $\psi (R)$. 
Hence, 
the lemma follows by the injectivity of $\psi $. 
\end{proof}

For $\w \in \Gamma ^n$ 
and $(0,\ldots ,0)\neq (f_1,\ldots ,f_l)\in \kx ^l$ 
with $l\geq 1$, 
we set 
$$
\delta =\max \{ \degw f_i\mid i=1,\ldots ,l\} 
\quad\text{and}\quad 
I=\{ i\mid \degw f_i=\delta \} . 
$$
Then, 
we have the following lemma, 
which can be verified easily.

\begin{lem}\label{lem:IP}
If $\sum _{i\in I}f_i^{\w }\neq 0$, 
then we have 
$(f_1+\cdots +f_l)^{\w }=\sum _{i\in I}f_i^{\w }$. 
\end{lem}

Now, 
let $\psi :\kx \to \kxb $ 
be a homomorphism of $k$-algebras 
with $\psi (x_i)\neq 0$ for each $i$. 
For $\uu \in \Gamma ^m$, 
we define a homomorphism 
$\psi ^{\uu }:\kx \to \kxb $ 
of $k$-algebras by 
$\psi ^{\uu }(x_i)=\psi (x_i)^{\uu }$ 
for $i=1,\ldots ,n$. 
Set 
$$
\uu _{\psi }=(\degu \psi (x_1),\ldots ,\degu \psi (x_n))
\in \Gamma ^n. 
$$
With this notation, 
the following proposition holds.

\begin{prop}\label{prop:IP}
If $\psi ^{\uu }(f^{\uu _{\psi }})\neq 0$ 
for $f\in \kx $, 
then we have 
$\psi (f)^{\uu }=\psi ^{\uu }(f^{\uu _{\psi }})$. 
\end{prop}
\begin{proof}
Write $f$ as in (\ref{eq:f}), 
and set $f_i=u_ix_1^{i_1}\cdots x_n^{i_n}$ 
for each $i=(i_1,\ldots ,i_n)$. 
Then, 
we have $f=\sum _if_i$ and $\psi (f)=\sum _i\psi (f_i)$. 
We apply Lemma~\ref{lem:IP} to 
$(\psi (f_i))_{i\in \supp (f)}$. 
Note that $\degu \psi (f_i)=i\cdot \uu _{\psi }$ 
and $\psi (f_i)^{\uu }=\psi ^{\uu }(f_i)$ 
for each $i\in \supp (f)$ by (\ref{eq:deg fg}). 
Hence, 
we get 
$$
\delta 
=\max \{ \degu \psi (f_i)\mid i\in \supp (f)\} 
=\max \{ i\cdot \uu _{\psi }\mid i\in \supp (f)\} 
=\deg _{\uu _{\psi }}f, 
$$
and so 
$$
I=\{ i\mid \degu \psi (f_i)=\delta \} 
=\{ i\mid i\cdot \uu _{\psi }=\deg _{\uu _{\psi }}f\} . 
$$ 
Thus, 
we have $\sum _{i\in I}f_i=f^{\uu _{\psi }}$. 
Therefore, we know that 
$$
\sum _{i\in I}\psi (f_i)^{\uu }
=\sum _{i\in I}\psi ^{\uu }(f_i)
=\psi ^{\uu }\left( \sum _{i\in I}f_i\right) 
=\psi ^{\uu }(f^{\uu _{\psi }})\neq 0. 
$$
By Lemma~\ref{lem:IP}, 
it follows that 
$\psi (f)^{\uu }=\left(\sum _i\psi (f_i)\right)^{\uu }$ 
is equal to the left-hand side of the preceding equality, 
and hence to $\psi ^{\uu }(f^{\uu _{\psi }})$. 
\end{proof}

Fix any $1\leq l\leq m$. 
Let $\kx ^{\psi }$ and $\kx ^{\psi ^{\uu }}$ 
be the $k$-subalgebras of $\kx $ defined as the 
inverse images of $k[x_1,\ldots ,x_l]$ 
by $\psi $ and $\psi ^{\uu }$, respectively. 
Then, 
we have the following theorem.

\begin{thm}\label{thm:IP}
\noindent{\rm (i)} 
$f^{\uu _{\psi }}$ belongs to 
$\kx ^{\psi ^{\uu }}$ 
for each $f\in \kx ^{\psi }$. 

\noindent{\rm (ii)} 
Let $S\subset \kx \sm \zs $ 
be such that $\trd _kk[S]=n$. 
If $\psi ^{\uu }$ is injective 
and $\psi ^{\uu }(\kx )$ 
is not contained in $k[x_1,\ldots ,x_l]$, 
then there exists $g\in S$ 
such that $f^{\uu _{\psi }}$ 
is not divisible by $g$ 
for any $f\in \kx ^{\psi}\sm \zs$. 
\end{thm}
\begin{proof}
(i) 
Since $\kx ^{\psi ^{\uu }}$ 
is the inverse image of $k[x_1,\ldots ,x_l]$ 
by $\psi ^{\uu }$, 
it suffices to check that 
$\psi ^{\uu }(f^{\uu _{\psi }})$ 
belongs to $k[x_1,\ldots ,x_l]$. 
This is clear if $\psi ^{\uu }(f^{\uu _{\psi }})=0$. 
If $\psi ^{\uu }(f^{\uu _{\psi }})\neq 0$, 
then we have $\psi ^{\uu }(f^{\uu _{\psi }})=\psi (f)^{\uu }$ 
by Proposition~\ref{prop:IP}. 
This is an element of 
$k[x_1,\ldots ,x_l]$, 
since so is $\psi (f)$ 
by the choice of $f$.

(ii) 
Since $\psi ^{\uu }$ is injective by assumption, 
$\kx ^{\psi ^{\uu }}$ is algebraically closed 
and factorially closed in $\kx $ by Lemma~\ref{lem:closed}. 
Since $\psi ^{\uu }(\kx )$ 
is not contained in $k[x_1,\ldots ,x_l]$, 
we have $\kx ^{\psi ^{\uu }}\neq \kx $. 
Hence, 
$\trd _k\kx ^{\psi ^{\uu }}$ is less than $n$. 
Since $\trd _kk[S]=n$ by assumption, 
we may find $y_1,\ldots ,y_n\in S$ 
such that $\trd _kk[y_1,\ldots ,y_n]=n$. 
Suppose that the assertion is false. 
Then, 
for each $1\leq i\leq n$, 
there exists 
$f_i\in \kx ^{\psi }\sm \zs $ 
such that $f_i^{\uu _{\psi }}$ is divisible by $y_i$. 
Then, 
$(f_1\cdots f_n)^{\uu _{\psi }}$ 
belongs to $\kx ^{\psi ^{\uu }}$ by (i), 
and is divisible by $y_1,\ldots ,y_n$ due to (\ref{eq:deg fg}). 
Since $\kx ^{\psi ^{\uu }}$ 
is factorially closed in $\kx $, 
it follows that $y_1,\ldots ,y_n$ 
belong to $\kx ^{\psi ^{\uu }}$, 
a contradiction. 
\end{proof}

Now, 
let us complete the proof of Theorem~\ref{thm:main}. 
Note that $\Gamma $ is torsion-free 
due to the structure of total ordering. 
Hence, 
we may regard $\Gamma $ 
as a subgroup of $\Q \otimes _{\Z }\Gamma $ 
which also has a structure of totally ordered additive group 
induced from $\Gamma $. 
Write 
$$
\phi (x_i)=p_i(z)=\sum _{j\geq 0}p_{i,j}z^j
$$
for $i=1,\ldots ,n$, 
where $p_{i,j}\in \kxb $ for each $j$. 
Since $\phi (\kx )$ 
is not contained in $\kxb $ by assumption, 
we have $p_{i,j}\neq 0$ 
for some $1\leq i\leq n$ and $j\geq 1$. 
Following \cite{DHM}, 
we define 
$\uu =(\vv,-\degv \phi )
\in (\Q \otimes _{\Z }\Gamma )^{m+1}$, 
where 
$$
\degv \phi :=\max \left\{ 
\frac{1}{j}\left(
\degv p_{i,j}-w_i\right)
\Bigm| i=1,\ldots ,n,\ j\geq 1 \right\} . 
$$
We show that $\uu _{\phi }=\w $, 
$\phi ^{\uu }$ is injective, 
and $\phi ^{\uu }(\kx )$ is not contained in $\kxb $. 
Then, 
the proof is completed by Theorem~\ref{thm:IP} (ii).

By the maximality of $\degv \phi $, 
we have 
$\degv p_{i,j}-w_i\leq j\degv \phi $, 
and so 
$$
\degu p_{i,j}z^j=\degv p_{i,j}-j\degv \phi \leq w_i 
$$
for each $1\leq i\leq n$ and $j\geq 1$. 
Since 
$\degu p_{i,0}z^0=\degv p_{i,0}=\degv p_i(0)=w_i$, 
it follows that 
$\degu \phi (x_i)
=\max \{ \degu p_{i,j}z^j\mid j\geq 0\} =w_i$ 
for each $i$. 
This proves that $\uu _{\phi }=\w $. 
Let $J(i)$ be the set of $j\geq 1$ 
such that $\degu p_{i,j}z^j=w_i$ 
for each $i$. 
Then, 
we have 
$$
\phi ^{\uu }(x_i)=\phi (x_i)^{\uu }
=\sum _{j\in J(i)}p_{i,j}^{\vv }z^j+p_i(0)^{\vv}. 
$$ 
Since $p_i(0)^{\vv}$'s 
are algebraically independent over $k$, 
we see that $\phi ^{\uu }(x_i)$'s 
are algebraically independent over $k$. 
Therefore, 
$\phi ^{\uu }$ is injective. 
By the definition of $\degv \phi $, 
there exist $1\leq i_0\leq n$ 
and $j_0\geq 1$ such that 
$\degv p_{i_0,j_0}-w_{i_0}=j_0\degv \phi $. 
Then, 
we have 
$\degu p_{i_0,j_0}z^{j_0}=w_{i_0}$. 
Hence, 
the monomial 
$z^{j_0}$ appears in $\phi ^{\uu }(x_{i_0})$ 
with coefficient $p_{i_0,j_0}^{\vv }\neq 0$. 
Thus, 
$\phi ^{\uu }(x_{i_0})$ 
does not belong to $\kxb $. 
Therefore, 
$\phi ^{\uu }(\kx )$ is not contained in $\kxb $. 
This completes the proof of Theorem~\ref{thm:main}.

\section{Remarks on stable coordinates 
and stable $\Ga $-invariants}
\label{sect:rmk}
\setcounter{equation}{0}

Shpilrain-Yu~\cite{SY} remarked that 
every stable coordinate of $\C [\x ]$ 
is a coordinate of $\C [\x ]$ for $n=2,3$. 
In the case of $n=2$, 
their proof is based on 
the theorem of Abhyankar-Moh~\cite{AM} and Suzuki~\cite{Suzuki}, 
and the cancellation theorem of 
Abhyankar-Heinzer-Eakin~\cite{AEH}. 
Hence, the result is valid not only for $\C $, 
but also for any field of characteristic zero. 
By Proposition~\ref{prop:steadfast} below, 
the statement holds 
for a more general class of commutative rings. 
Recall that a commutative ring $k$ with identity is said to be 
{\it steadfast} if 
the following condition holds 
for any commutative ring $A$ containing $k$ 
(cf.~\cite{Hamann}): 
{\it If the polynomial rings 
$k[x_1,\ldots ,x_n]$ and $A[x_2,\ldots ,x_n]$ 
are $k$-isomorphic for some $n\geq 1$, 
then $k[x_1]$ and $A$ are $k$-isomorphic. }

The following remark is due to Amartya K. Dutta 
who answered the author's question 
on his visit to Indian Statistical Institute 
in 2013 (see also~\cite{BD}).

\begin{prop}\label{prop:steadfast}
Let $k$ be a commutative ring with identity 
such that $k[x_1]$ is steadfast. 
If $f\in k[x_1,x_2]$ is a coordinate of 
$k[x_1,\ldots ,x_n]$ for some $n\geq 3$, 
then $f$ is a coordinate of $k[x_1,x_2]$. 
\end{prop}
\begin{proof}
Put $A=k[x_1,x_2]$ and $k'=k[f]$. 
Then, 
there exist 
$f_2,\ldots ,f_n\in k[x_1,\ldots ,x_n]$ 
such that 
$A[x_3,\ldots ,x_n]=k'[f_2,\ldots ,f_n]$. 
Hence, 
the polynomial ring 
$k'[y_2,\ldots ,y_n]$ 
over $k'$ 
is $k'$-isomorphic to $A[x_3,\ldots ,x_n]$ 
via the isomorphism defined by 
$y_i\mapsto f_i$ for each $i$. 
Since $k'\simeq k[x_1]$ is steadfast by assumption, 
it follows that 
$k'[y_2]$ and $A$ are $k'$-isomorphic. 
Thus, 
we have $A=k'[g]=k[f,g]$ for some $g\in A$. 
Therefore, 
$f$ is a coordinate of $A$. 
\end{proof}

For example, 
integrally closed domains are steadfast 
due to Asanuma~\cite{Asanuma} 
(see also \cite{AEH} and \cite{Hamann}). 
If $k$ is an integrally closed domain, 
then $k[x_1]$ is also  an integrally closed domain. 
Hence, 
every stable coordinate of 
$k[x_1,x_2]$ is a coordinate of $k[x_1,x_2]$ 
by Proposition~\ref{prop:steadfast}. 
On the other hand, 
Neena Gupta pointed out that 
Bhatwadekar-Dutta~\cite[Example 4.1]{BD} 
constructed a ``residual variable" 
of $k[x_1,x_2]$ which is not a coordinate of $k[x_1,x_2]$ 
when $k=\Z _{(2)}[2\sqrt{2}]$. 
Here, 
for a commutative noetherian ring $k$, 
the notion of residual variable of $k[x_1,x_2]$ 
is equivalent to the notion of 
stable coordinate of $k[x_1,x_2]$ 
(cf.~\cite[Theorem A]{BD}). 
Therefore, 
not every stable coordinate of $k[x_1,x_2]$ 
is a coordinate of $k[x_1,x_2]$.

In the case of $n=3$, 
Shpilrain-Yu used the cancellation theorem 
of Miyanishi-Sugie~\cite{MS} and Fujita~\cite{Fujita}, 
and the result of Kaliman~\cite{Kaliman} 
to show that every stable coordinate of $\C [\x ]$ 
is a coordinate of $\C [\x ]$. 
Neena Gupta remarked that 
one can prove a similar statement 
over any field of characteristic zero 
by using the results of Sathaye~\cite{Sathaye} 
and Bass-Connell-Wright~\cite{BCW} 
instead of \cite{Kaliman}.

Next, 
we discuss stable $\Ga $-invariants of $\kx $. 
In what follows, we assume that 
$k$ is a field of characteristic zero. 
Then, 
for a $k$-subalgebra $A$ of $\kx $, 
we have $A=\kx ^{\Ga }$ for some $\Ga $-action on $\kx $ 
if and only if $A=\ker (D)$ 
for some locally nilpotent $k$-derivation of $\kx $. 
The following result~\cite[Corollary 5.40]{Fbook} 
(see also \cite{F}) is a corollary to 
\cite[Theorem 5.37]{Fbook} 
which is due to Daigle and Freudenburg.

\begin{prop}\label{prop:DF}
Assume that $n\geq 2$. 
Let $D$ be a locally nilpotent $k$-derivation 
of $\kx $ with $k[x_1,x_2]\cap \ker (D)\neq k$. 
Then, 
we have either $D(x_1)=D(x_2)=0$, 
or $k[x_1,x_2]\cap \ker (D)\subset k[g]\subset \ker D$ 
for some coordinate $g$ of $k[x_1,x_2]$. 
\end{prop}

In the situation of Proposition~\ref{prop:DF}, 
we have $k[x_1,x_2]\cap \ker (D)=k[g]$ 
unless $D(x_1)=D(x_2)=0$, 
since $k[x_1,x_2]\cap \ker (D)$ and $k[g]$ 
are both algebraically closed in $k[x_1,x_2]$ 
and of transcendence degree one over $k$. 
Hence, 
for any $\Ga $-action on $\kx $ 
with $k[x_1,x_2]\cap \kx ^{\Ga }$ 
not equal to $k$ or $k[x_1,x_2]$, 
there exists a coordinate $g$ of $k[x_1,x_2]$ 
such that $k[x_1,x_2]\cap \kx ^{\Ga }=k[g]$.

Now, 
we show that every stable $\Ga $-invariant of 
$k[x_1,x_2]$ is an invariant 
for a nontrivial $\Ga $-action on $k[x_1,x_2]$. 
Take any stable $\Ga $-invariant 
$f$ of $k[x_1,x_2]$ not belonging to $k$. 
By definition, 
there exists a $\Ga $-action on $\kx $ 
such that $\kx ^{\Ga }$ contains $f$, 
but does not contain $k[x_1,x_2]$. 
Then, 
$A:=k[x_1,x_2]\cap \kx ^{\Ga }$ 
is not equal to $k$ or $k[x_1,x_2]$. 
Hence, 
we have $A=k[g]$ 
for some coordinate $g$ of $k[x_1,x_2]$ 
as remarked. 
As mentioned after Theorem~\ref{thm:coords}, 
$k[g]=k[x_1,x_2]^{\Ga }$ holds 
for some nontrivial $\Ga $-action on $k[x_1,x_2]$. 
Since $f$ is an element of 
$A=k[g]=k[x_1,x_2]^{\Ga }$, 
we know that $f$ 
is an invariant for this action of $\Ga $.

\noindent
Department of Mathematics and Information Sciences\\ 
Tokyo Metropolitan University \\
1-1  Minami-Osawa, Hachioji \\
Tokyo 192-0397, Japan\\
kuroda@tmu.ac.jp

\end{document}